\documentclass[12pt,a4paper]{scrartcl}
\usepackage[T1]{fontenc}
\usepackage[english]{babel}	
\usepackage{german}

\usepackage{graphicx}
\usepackage[a4paper]{geometry}
\usepackage{fullpage}
\usepackage{hyperref}

\usepackage{color}
\usepackage{colortbl}
\definecolor{lightgrey}{rgb}{0.85,0.875,0.85}   
\newcommand{\clg}{\cellcolor{lightgrey}}

\usepackage{algorithm}
\usepackage{algpseudocode} 
\usepackage{amsmath}
\usepackage{amsthm}
\usepackage{amssymb}
\usepackage{mathtools}
\usepackage{mathpazo}
\usepackage{array}
\usepackage{nicefrac}

\theoremstyle{definition}
\newtheorem{defi}{Definition}[section]
\newtheorem*{rema}{Remark}	
\newtheorem{conj}[defi]{Conjecture}

\theoremstyle{plain}

\newtheorem{prop}[defi]{Proposition}
\newtheorem{coro}[defi]{Corollary}
\newtheorem{lemm}[defi]{Lemma}

\numberwithin{equation}{section}

\newcommand{\mN}{\mathbb{N}}
\newcommand{\mZ}{\mathbb{Z}}
\newcommand{\mP}{\mathbb{P}}

\newcommand{\mdot}{\!\cdot\!}
\newcommand{\ndiv}{\!\nmid\!}

\newcommand{\copr}{\!\perp\!}

\newcommand{\ap}[2]{\langle#1\rangle_{#2}}

\DeclareMathOperator{\modulo}{mod} 
\newcommand{\modo}[1]{\modulo #1}
\newcommand{\modu}[1]{\ (\modulo #1)}
\DeclareMathOperator{\congr}{\equiv}
\DeclareMathOperator{\ncongr}{\not\equiv}

\newcommand{\vl}{\vrule width 2pt}
\newcommand{\lb}{\linebreak}

\begin{document}

\title{On differences between consecutive numbers coprime to primorials}
\author{Mario Ziller}
\date{}

\maketitle

\begin{abstract}

We consider the ordered sequence of coprimes to a given primorial\lb number and investigate differences between consecutive elements. The\lb Jacobsthal function applied to the concerning primorial turns out to\lb represent the greatest of these differences. We will explore the smallest even number which does not occur as such a difference.

Little is known about even natural numbers below the respective\lb Jacobsthal function which cannot be represented as a difference between consecutive numbers coprime to a primorial. Existence and frequency of these numbers have not yet been clarified.

Using the relation between restricted coverings of sequences of consecutive integers and the occuring differences, we derive a bound below which all even natural numbers are differences between consecutive numbers\lb coprime to a given primorial $p_k\#$. Furthermore, we provide exhaustive computational results on non-existent differences for primes $p_k$ up to\lb $k=44$. The data suggest the assumption that all even natural numbers up to $h(k-1)$ occur as differences of coprimes to $p_k\#$ where $h(n)$ is the Jacobsthal function applied to $p_n\#$.

\end{abstract}

\section{Introduction and general definitions}

\subsection*{Notation}

Henceforth, we denote the set of integral numbers by $\mZ$ and the set of natural\lb numbers, i.e. positive integers, by $\mN$. $\mP=\{p_i\mid i\in\mN\}$ is the ordered set of prime\lb numbers with $p_1=2$. For $n,m\in\mZ$, we abbreviate $(n,m)=\gcd(n,m)$.

 As usual, we define the $k^{th}$ primorial number as the product of the first  $k$ primes:
 $$p_k\#=\prod_{i=1}^k p_i\ , k\in\mN.$$

\subsection*{Coprimes to primorial numbers and their differences}

In an analysis of the sieve of Eratost\'{e}nes, de Polignac investigated the numbers in the sieve remaining after each step, i.e. the ordered sequence of coprime numbers to primorials \cite{de_Polignac_1849_II}. He concluded several properties of this sequence and of the differences of consecutive numbers within it.

\begin{defi}  {\itshape Sequences of coprimes and its differences.} \label{copr-diff}\\
Let  $k\in\mN$. The sequence of coprimes to the $k^{th}$ primorial is the ordered sequence
$C(k)=\left(c_{k,i}\right)_{i=1,\dots,\infty}$ where $\left\{c_{k,i}\right\}_{i=1}^\infty=\left\{x\in\mN \mid (x,p_k\#)=1 \right\}$ and $c_{k,i}<c_{k,j}$ for $i<j$.

We define the sequence of differences between consecutive coprimes to $p_k\#$ as\lb
$D(k)=\left(d_{k,i}\right)_{i=1,\dots,\infty}=\left(c_{k,i+1}-c_{k,i}\right)_{i=1,\dots,\infty}$.
\end{defi}

Two basic statements can be derived by simple considerations.

\begin{lemm} \label{N-min}
For every $k\in\mN$, there is an even $N_{min}(k)\in\mN$ such that $2\mdot n\in D(k)$ for all $n\in\mN\le N_{min}(k)/2$. Moreover, $N_{min}(k)\ge 2$.
\end{lemm}

\begin{proof}
All elements of $C(k)$ are odd numbers because $p_k\#$ is an even number. Thus, all elements of $D(k)$ are even numbers. $2\in D(k)$ for all $k\in\mN$, i.e. the difference 2 always occurs as a difference of consecutive coprimes because $p_k\#-1,p_k\#+1\in C(k)$. So, $N_{min}(k)\ge 2$.
 \end{proof}

\begin{lemm} \label{N-max}
For every $k\in\mN$, there is an even $N_{max}(k)\in\mN$ such that $2\mdot n\not\in D(k)$ for all $n\in\mN>N_{max}(k)/2$. Moreover, $N_{max}(k)<p_k\#$.
\end{lemm}

\begin{proof}
 Both sequences, $C(k)$ and $D(k)$, are periodic because $a \congr a+p_k\# \modu {p_k\#}$ for any $a\in\mN$ and $p_k\in\mP$. The smallest period of $C(k)$ and $D(k)$ is also $p_k\#$. Let $x\in\mN$, and $a \congr a+x \modu {p_k\#}$ for all $a\in\mN$. Then $p_k\#$ divides $x$, and therefore $x\ge p_k\#$. Thus, $N_{min}(k)\le N_{max}(k)<p_k\#$ because $1,p_k\#-1\in C(k)$ .
\end{proof}

De Polignac used these relationships and even presumed $N_{min}(k)\ge 2\mdot p_{k-1}$\cite{de_Polignac_1849_II} for $k>1$. But no evidence has been handed down for that.

\subsection*{Sequences and Jacobsthal's function}

De Polignac also demonstrated the general relationship between differences of\lb consecutive coprimes and sequences of natural numbers \cite{de_Polignac_1849_I}. If two natural numbers $x$ and $y$ are consecutive coprimes to a given $n\in\mN$, then all integers between $x$ and $y$ must have a common divisor with $n$.

We introduce a shortened notation for sequences of consecutive integers. Please note that $a$ is the first member of $\ap a m$, in contrast to a previous definition \cite{Ziller_2019}.

\begin{defi}  {\itshape Sequences of consecutive integers.}\\
Let  $m\in\mN$ and $a\in\mZ$. The finite sequence of consecutive integers\\
$\left(a_i\right)_{i=1,\dots ,m} = \left(a+j\right)_{j=0,\dots ,m-1}$ is denoted by $\ap a m$.
\end{defi}

There is a close relationship to the Jacobsthal function. This function $j(n)$ has been known to be the smallest positive integer $m$, such that every sequence of $m$ consecutive integers contains at least one integer coprime to $n$ \cite{Jacobsthal_1960_I, Erdoes_1962}.

The specific case $n=p_k\#$ leads to the smallest positive integer $m$, such that every sequence of $m$ consecutive integers contains an integer coprime to the product of the first $k$ primes \cite{ Hagedorn_2009, Costello_Watts_2015,Ziller_Morack_2016}.

\begin{defi} {\itshape Primorial Jacobsthal function.} \label{h}\\
For $k\in\mN$, the primorial Jacobsthal function $h(n)$ is defined as
$$h(k)=\min\ \{m\in\mN\mid\forall\ a\in\mZ\ \exists\ j\in\{0,\dots,m-1\}:a+j\copr p_k\#\}.$$
\end{defi}

In other words, there exists a sequence of at most $m-1$ consecutive integers such that all elements of it have a common divisor with $p_k\#$. $h(k)-1$ is the maximum length of a sequence with only elements that are not coprime to  $p_k\#$. By this conclusion, we can directly derive that $h(k)=N_{max}(k)$ as defined above.\\

The conjecture of de Polignac $N_{min}(k)\ge 2\mdot p_{k-1}$ for $k>1$, however, is related to\lb a known lower bound \cite{Kanold_1965,Hagedorn_2009} on Jacobsthal's function.

\begin{prop} \label{h-low}\ \\
Let $k\in\mN>1$. Then
$$h(k)\ge 2\mdot p_{k-1}.$$
\end{prop}

\begin{proof}
The sequence $\ap 2 3$ demonstrates the assertion for $k=2$. Let now $k>2$.

By the Chinese Remainder Theorem, there exists an $a\in\mZ$ satisfying the\lb simultaneous congruences $a\congr 0\modo p_{k-2}\#$, $a\congr 1\modo p_{k-1}$,  and $a\congr -1\modo p_k$. Then\lb $a\pm p_{k-1}\congr\pm p_{k-1}\modo p_{k-2}\#$, $a\pm p_{k-1}\congr 1\modo p_{k-1}$, and $a\pm p_{k-1}\congr -1\pm p_{k-1} \modo p_k$, respectively. Consequently, $a-p_{k-1}$ and $a+p_{k-1}$ are coprime to $p_k\#$.

On the other hand, $a$ and $a\pm j,\ j=2,\dots,p_{k-1}-1$ have common divisors\lb with $p_{k-2}\#$. With $p_{k-1}/(a-1)$ and $p_k/(a+1)$, all elements of the sequence\lb $\ap {a\!-\!p_{k-1}\!+\!1} {2*p_{k-1}-1}=(a\!-\!p_{k-1}\!+\!1,\dots,a\!+\!p_{k-1}\!-\!1)$ have common divisors with $p_k\#$.
\end{proof}

Within this proof, we have also shown that $a-p_{k-1},\ a+p_{k-1}\in C(k)$, and\lb therefore $2\mdot p_{k-1}\in D(k)$. Furthermore, we conclude $N_{min}(k)\le 2*p_{k-1}$ and\lb $N_{max}(k)\ge 2*p_{k-1}$ for $k>1$ by definition.

\subsection*{Coverings and coprime differences}

In this paper, we follow the concept of coverings \cite{Gassko_1996,Hajdu_Saradha_2012} which we have adapted to the problem under consideration. A set of residue classes $a_i\modo \pi_i,\ i=1,\dots,k$ is called covering of a sequence $\ap a m$ if each element $a+j,\ j=0,\dots,m-1$ of the sequence belongs to one of the residue classes. We shortly say $\{a_i\modo\pi_i\}_{i=1}^k$ covers $\ap a m$. The respective residue class covers $a+j$.

\begin{defi} {\itshape Covering.}\\
Let $a\in\mZ$,\ \ $m,k\in\mN$,\ and $\pi_i\in\mP,\ a_i\in\{0,\dots,\pi_i-1\},\ i=1,\dots,k$.

$\{a_i\modo\pi_i\}_{i=1}^k \text{ \,is a covering of } \ap a m$

\qquad$\iff\ \forall\ j\in\{0,\dots,m-1\}\ \exists\ i\in\{1,\dots,k\}:a+j\congr a_i\modo \pi_i.$
\end{defi}

\begin{rema}
The residue classes $a_i\modo\pi_i$ are uniquely related to $a$.

When a residue class $a_i\modo\pi_i$ covers two positions $x$ and $y$ of $\ap a m$ with $x\ne y$,\lb then $a+x\congr a+y\congr a_i\modo \pi_i$, and therefore $x\congr y\modo \pi_i$. We get  $a\congr a_i-x \modo \pi_i$\lb because $a+x\congr a_i-x+x\congr a_i\modo \pi_i$. $a$ is uniquely determined by the Chinese\lb Remainder Theorem.
\end{rema}
\ 

We show that every difference $m\in D(k)$ corresponds to a covering of a sequence\lb of length $m-1$.

\begin{prop} \label{coprime-covering}\ \\
Let $m,k\in\mN$. For every element $m$ of $D(k)$, there exists a covering of a sequence of\lb length $m-1$.\vspace{3.5mm}\\
$ m\in D(k) \Longrightarrow \exists\ a\in\mZ\ \exists\ a_i\in\{0,\dots,p_i-1\}, i=1,\dots,k: \{a_i\modo p_i\}_{i=1}^k \text{ covers } \ap a {m-1}. $
\end{prop}

\begin{proof}
By definition \ref{copr-diff}, $m \in D(k)$ if and only if

$\exists\ x\in\mN: (x,p_k\#)=1 \land (x+m,p_k\#)=1 \land (x+1+j,p_k\#)>1 \text{ \,for } j=0,\dots,m-2.$\\
Then $a=x+1$ and $a_i\congr 0\modo p_i$ meet the requirements.
 \end{proof}
\ 

In the following chapters, we will exploit properties of coverings to get insight into the structure of $D(k)$. We focus on even numbers below $h(k)$ which do not belong\lb to $D(k)$ and derive a lower bound on $N_{min}(k)$ whereas $N_{max}(k)=h(k)$ has always been proven above.\\

The existence of an even number $N_{min}(k)<2\mdot n<N_{max}(k)$, i.e. $2\mdot n\not\in D(k)$\lb has not been debated in this paper so far. There are several $k\in\mN$ with\lb $N_{min}(k)=N_{max}(k)=h(k)$. 
By definition, $N_{min}(k)<N_{max}(k)$ consequently means $N_{min}(k)\le N_{max}(k)-4$. The smallest $k$ with this property is $k=6$. We get $p_k=13$, $N_{min}(k)=18$, $N_{max}(k)=22$, and $20\not\in D(k)$.\\

The Greedy Permutation Algorithm \cite{Ziller_Morack_2016} has proved to be efficient to compute\lb Jacobsthal's function and related sequences. We extended and adapted this algorithm\lb for the determination of restricted coverings which represent the elements of the\lb sets $D(k)$.

In an exhaustive exploration for primes up to $k=44$, we computed $N_{min}(k)$ as\lb well as all non-existent differences $2\mdot n$ between consecutive coprimes to $p_k\#$ with $N_{min}(k)<2\mdot n<N_{max}(k)$. The data suggest the assumption that at least all even natural numbers up to $h(k-1)$ occur as differences of coprimes to $p_k\#$ for $k>1$.

\pagebreak

\section{Restricted coverings}

We recall the findings about D(k) that have been demonstrated so far, cf. lemma\ref{N-min}, proposition\ref{h-low}, and definition\ref{h}. The sequence $D(1)$ consists of the single element  $2$. For all $k\in\mN>1$, we have

\qquad$2,\ 2\mdot p_{k-1},\ h(k) \in D(k)$,

\qquad$2\le N_{min}(k)\le N_{max}(k)=h(k)$, \qquad and

\qquad$2*p_{k-1}\le N_{max}(k)=h(k)$.\\

Every even natural number $m\le h(k)$ is a potential element of $D(k)$. In this chapter, we intend to find out conditions for whether such an $m$ belongs to $D(k)$ or not. In proposition\ref{coprime-covering}, we proved that for each $m\in D(k)$ there exists a covering of a sequence of length $m-1$ by residue classes $\modo p_i,\ i=1,\dots,k$. Such coverings are now to be examined in more detail.

\begin{lemm} \label{a-->b}\ \\
Let $m,k\in\mN$, $a\in\mZ$, $\pi_i\in\mP$, and $a_i\in\{0,\dots,\pi_i-1\}$, $i=1,\dots,k$. Furthermore,\lb let $\{a_i\modo \pi_i\}_{i=1}^k$ be a covering of $\ap a m$.

For every $b\in\mZ$, there exists a covering $\{b_i\modo \pi_i\}_{i=1}^k$ of $\ap b m$, $b_i\in\{0,\dots,\pi_i-1\}$, $i=1,\dots,k$.
\end{lemm}

\begin{proof}
For all $i=1,\dots,k$, we set $b_i\congr a_i+b-a\modu{\pi_i}$.\\
Then $b_i\congr a_i+b-a\congr a+j+b-a\congr b+j\modo \pi_i$ for an appropriate $j\in\{0,\dots,m-1\}$ with $a+j\congr a_i\modo \pi_i$.
\end{proof}

A covering can be relocated if required. For our computations, this lemma was\lb applied specifically to $b=1$. We only processed coverings of $\ap 1 m$.\\

Another consequence of lemma \ref{a-->b} is the existence of an associated covering where all residue classes are zero if all primes are different. The distinctness of primes has not been requested in the general definition of coverings. A covering may also include different residue classes of the same prime number. However, the following lemma requires distinct primes.

\begin{coro} \label{a<-->0}\ \\
Let $m,k\in\mN$, $a\in\mZ$, $\pi_i\in\mP$, and $\pi_i\neq\pi_j$ for $i\neq j$, $\ i,j=1,\dots,k$. Furthermore,\lb let $a_i\in\{0,\dots,\pi_i-1\}$, and $\{a_i\modo \pi_i\}_{i=1}^k$ be a covering of $\ap a m$.\vspace{3.5mm}

There exists $b\in\mZ$ such that $\{0\modo \pi_i\}_{i=1}^k$ is a covering of $\ap b m$.
\end{coro}

\begin{proof}
There is a solution $b$ of the system of congruences
$b\congr a-a_i\modo \pi_i$, $i=1,\dots,k$ by the Chinese Remainder Theorem.

Then $b+j\congr a-a_i+j\congr 0\modo \pi_i$ for an appropriate $j\in\{0,\dots,m-1\}$ with\lb $a+j\congr a_i\modo \pi_i$.
\end{proof}
\ 

If  $\{a_i\modo \pi_i\}_{i=1}^k$ covers $\ap a m$, it can happen that it also covers $\ap {a-1} m$ or $\ap {a+1} m$, i.e. one of the residue classes covers $a-1$ or $a+m$ in addition. In such a case, we can use the left-most sequence as the default. After suitable relocation to  $\ap 1 m$, a covering can be selected, none of the residue classes of which is zero.

\begin{coro} \label{non-zero}\ \\
Let $m,k\in\mN$, $a\in\mZ$, $\pi_i\in\mP$, and $a_i\in\{0,\dots,\pi_i-1\}$, $i=1,\dots,k$. Furthermore,\lb let $\{a_i\modo \pi_i\}_{i=1}^k$ be a covering of $\ap a m$.\vspace{3.5mm}

There exists a covering $\{b_i\modo \pi_i\}_{i=1}^k$ of $\ap 1 m$ where $b_i\in\{1,\dots,\pi_i-1\}$.
\end{coro}
 
\begin{proof}
We set $a^*=a$. If $a^*-1$ was covered by any of the residue classes $a_i\modo \pi_i$, then we repeat $a^*=a^*-1$ until $a^*-1$ is uncovered. Then $\{a_i\modo \pi_i\}_{i=1}^k$ is also\lb a covering of $\ap {a^*} m$ with $a_i\ncongr a^*-1\modo \pi_i$ for all $i=1,\dots,k$.

By lemma\ref{a-->b}, there exists a covering $\{b_i\modo \pi_i\}_{i=1}^k$ of $\ap 1 m$ where\lb $b_i\congr a_i+1-a^*\ncongr 0\modu{\pi_i}$, i.e. $b_i\in\{1,\dots,\pi_i-1\}$.
\end{proof}

A covering $\{a_i\modo \pi_i\}_{i=1}^k$ of $\ap a m$ which does not cover $\ap {a-1} m$ or $\ap {a+1} m$,\lb we call restricted covering. None of its residue classes cover $a-1$ or $a+m$ in addition.

\begin{defi} {\itshape Restricted covering.} \label{restricted}\\
Let $m,k\in\mN$, $a\in\mZ$, $\pi_i\in\mP$, and $a_i\in\{0,\dots,\pi_i-1\}$, $i=1,\dots,k$. Furthermore,\lb let $\{a_i\modo \pi_i\}_{i=1}^k$ be a covering of $\ap a m$.

$ \{a_i\modo\pi_i\}_{i=1}^k$ is a restricted covering of $\ap a m$ if

\qquad$\forall\ j\in\{0,\dots,m-1\}\ \exists\ i\in\{1,\dots,k\}:a+j\congr a_i\modo \pi_i$, and additionally

\qquad$\forall\ i\in\{1,\dots,k\}:\ a-1\ncongr a_i\modo \pi_i \land a+m\ncongr a_i\modo \pi_i$.\\
\end{defi}

In analogy to proposition \ref{coprime-covering}, we prove that every difference between two\lb consecutive coprimes to a product of different primes corresponds to an appropriate restricted covering of a sequence starting with $1$.

\begin{coro} \label{diff<-->rcov}\ \\
Let $k\in\mN$, $\pi_i\in\mP$, $\pi_i\neq\pi_j\ $ for $\ i\neq j$, $\ i,j=1,\dots,k$, and $\Pi=\prod_{i=1}^k \pi_i$.

There exist two consecutive coprime numbers $x<y$ to $\Pi$ if and only if there exists\lb a restricted covering $\{a_i\modo \pi_i\}_{i=1}^k$ of $\ap 1 {y-x-1}$.\vspace{3.5mm}

\quad$\exists\ x,y\in\mZ\ :\ (x,\Pi)=(y,\Pi)=1 \land (z,\Pi)>1\ $ for $\ x<z\in\mZ<y$

\quad\qquad$\iff$

\quad$\exists\ a_i\in\{1,\dots,\pi_i-1\}$, $i=1,\dots,k\ :\ $

\quad\qquad$\forall\ j\in\{0,\dots,y-x-2\}\ \exists\ i\in\{1,\dots,k\}:1+j\congr a_i\modo \pi_i$, and additionally

\quad\qquad$\forall\ i\in\{1,\dots,k\}:\ y-x\ncongr a_i\modo \pi_i$.
\end{coro}

\begin{proof}
Given $x,y\in\mZ$. Then $(x,\Pi)=1$, $(y,\Pi)=1$, and $(x+1+j,\Pi)>1$\lb for $j=0,\dots,y-x-2$ is equivalent to $\{0\modo \pi_i\}_{i=1}^k$ is a restricted covering\lb of $\ap {x+1} {y-x-1}$. Then $a_i\congr -x \modo \pi_i$ satisfy the requirements.
\end{proof}

Considering $\pi_i=p_i$, we can conclude that $ m\in D(k)$ is equivalent to the existence of a restricted covering $\{a_i\modo p_i\}_{i=1}^k$ of $\ap 1 {m-1}$. This is the basic idea of the following inferences and of our computations described in the next section.

We now prove the existence of a covering for any $k$ and the propagation of a\lb covering to greater primorials.\\

\begin{prop} {\itshape Existence of coverings.} \label{exist-cov}\\
For every $k\in\mN$, there exists a restricted covering $\{a_i\modo p_i\}_{i=1}^k$ of $\ap 1 {2\,\mdot\,k-1}$, i.e.
$$\forall k\in\mN\ :\ 2\mdot k\in D(k).$$
\end{prop}

\begin{proof}
\ We set $a_1\congr 1\modo 2$. Then $a_1\ncongr0\modo2$ and $a_1\ncongr2\mdot k\modo2$ whereas\lb $a_1\congr2\mdot j+1\modo2$ for $j\in\{0,\dots,k-1\}$.

The case $k=1$ was proven thereby. Let now be $k>1$. The remaining positions\lb of  $\ap 1 {2\,\mdot\,k-1}$ to be covered are $2\mdot j\ncongr a_1\modo2$, $j=1,\dots,k-1$.

We will successively select suitable residue classes $a_i\modo p_i$, $i=2,\dots,k$ and define\lb a bijection $f:\{a_i\modo p_i\}_{i=1}^k \mapsto \{2\mdot j\}_{j=0}^{k-1}$ between the covering and the remaining\lb sequence positions and zero where $f(a_1\modo 2)=0$. This ensures the covering\lb without taking care that a residue class $a_i\modo p_i$, $1<i\le k$ might cover more than one position.\\

Given $1\le x\in\mN<k$, $\ a_i\modo p_i$ were properly selected for $i=1,\dots,x$, and $f(a_i\modo p_i)\le2\mdot x$. All of these assumptions are satisfied for $x=1$.

The $x$ residue classes were mapped to $0$, and to $x-1$ elements of $\{2\mdot j\}_{j=1}^{x}$ if $x>1$. Thus, one position $2\mdot y$, $1\le y\le x$ is left where $y\ncongr 0$, $x+1\ncongr 0$, and\lb $y\ncongr x+1\ \modo p_{x+1}$ because $p_{x+1}>x+1>y$. $y$ and $x+1$ cannot belong to the same residue class $\modo p_{k+1}$. So, we can select $a_{x+1}\modo p_{x+1}$ such that $a_{x+1}\ncongr0\modo p_{x+1}$ and $a_{x+1}\ncongr2\mdot k\modo p_{x+1}$ because $p_{x+1}>3$. Either $a_{x+1}\congr y \text{ or } a_{x+1}\congr x+1\ \modo p_{x+1}$ meets the requirements. We assign $a_{x+1}\modo p_{x+1}$ to $2\mdot y$ or $2\mdot x+2$, respectively.

In case of $x=k-1$, one position $2\mdot y$, $1\le y\le k-1$ is left again with $y\ncongr 0$, and $y\ncongr k\ \modo p_k$ because $p_k>k$. We select $a_k\congr y\modo p_k$ and assign it to $2\mdot y$.

Hence, $\{a_i\modo p_i\}_{i=1}^{k}$ is a restricted covering of $\ap 1 {2\,\mdot\,k-1}$, and $2\mdot k\!\in\!\!D(k)$ by\lb corollary \ref{diff<-->rcov}.
\end{proof}
\ 

So we know that every even natural number occurs as the difference between\lb two consecutive coprimes to at least one primorial. Below, we demonstrate that\lb every difference between two consecutive coprimes to a primorial is also the difference between two consecutive coprimes to a greater primorial. An analogous statement has already been proven by de Polignac\cite{de_Polignac_1849_II}.

\begin{prop} {\itshape Propagation of coverings.} \label{k-->k+1}\\
Let $m,k\in\mN$. Then
$$m\in D(k)\ \Longrightarrow m\in D(k+1).$$
\end{prop}

\begin{proof}
According to corollary \ref{diff<-->rcov}, $m\in D(k)$ if and only if there exists a restricted covering $\{a_i\modo p_i\}_{i=1}^k$ of $\ap 1 {m-1}$, i.e.\ 

$\exists\ a_i\in\{1,\dots,p_i-1\},\ i=1,\dots,k\ :\ $

\qquad$\forall\ j\in\{0,\dots,m-2\}\ \exists\ i\in\{1,\dots,k\}:1+j\congr a_i\modo p_i$, and additionally

\qquad$\forall\ i\in\{1,\dots,k\}:\ m\ncongr a_i\modo p_i$.\\
We can select $a_{k+1}\in\{1,\dots,p_{k+1}-1\}$ such that $a_{k+1}\ncongr m\modo p_{k+1}$ because $p_{k+1}\ge3$. Then $\{a_i\modo p_i\}_{i=1}^{k+1}$ is also a restricted covering of $\ap 1 {m-1}$ and $m\in D(k+1)$.
\end{proof}
\ 

Propositions \ref{exist-cov} and \ref{k-->k+1} together state that the first $k$ even numbers are differences of consecutive coprimes to $p_k\#$. Thus, $2\mdot k$ is a lower bound on $N_{min}(k)$.

\begin{coro} {\itshape Unboundedness of $N_{min}$.} \label{lower-bound}\\
Let $k\in\mN$. Then
$$2\mdot k\le N_{min}(k).$$
\end{coro}

\begin{proof}
Proof by induction.  $2\le N_{min}(1)$ because $2$ is the single element of $D(1)$.

Given $2\mdot j\le N_{min}(j)$. Then $2\mdot j\le N_{min}(j+1)$ by proposition\ref{k-->k+1} and\lb $2\mdot {(j+1)}\in D(j+1)$ by proposition\ref{exist-cov}. 
\end{proof}
\

Summarising the results so far, we attain

\qquad$2\mdot k\le N_{min}(k)\le N_{max}(k)=h(k)$, \qquad and

\qquad$2\mdot p_{k-1}\in D(k)$ for all $k>1$.\\
In the next chapter, we will determine even numbers $N_{min}(k)<2\mdot n<N_{max}(k)$, that cannot be represented as a difference between consecutive coprimes to $p_k\#$. 

\pagebreak

\section{Computation}

Corollary\ref{diff<-->rcov} demonstrates that every difference $m$ between two consecutive\lb coprimes to $p_k\#$ corresponds to a restricted covering of $\ap 1 {m-1}$. We aim to develop a computational approach to the search for such restricted coverings.\\

The prime number 2 plays a specific role in many contexts. With respect to\lb coverings, we formulate a general lemma which makes the separation of the only\lb even prime consistent. Extending a set of odd primes by 2 can more than double the length of the coverable sequence. Analogous results have been described by various authors \cite{Hagedorn_2009,Hajdu_Saradha_2012,Ziller_Morack_2016,Ziller_2019}. 

\begin{lemm} \label{two}\ \\
Let $m,k\in\mN$, $k\ge2$, $\pi_1=2$, $\pi_i\in\mP>2$, and $\pi_i\neq\pi_j$ for $i\neq j$, $\ i,j=2,\dots,k$. Furthermore, let $b_1\in\{0,1\}$, and $\ a_i,b_i\in\{0,\dots,\pi_i-1\}$ for $i=2,\dots,k$.
$$ \exists\ a\in\mZ: \{a_i\modo\pi_i\}_{i=2}^k \text{ \,covers } \langle a\rangle_{m}\ \iff\ \exists\ b\in\mZ: \{b_i\modo\pi_i\}_{i=1}^k \text{ \,covers } \langle b\rangle_{2m+1}. $$
\end{lemm}

\begin{proof}\ \\
($\Rightarrow$)
Given $\forall\ x\in\{0,\dots,m-1\}\ \exists\ i\in\{2,\dots,k\}:a_i\equiv a+x\modu {\pi_i}$.

There exists $b\in\mZ$ with $b\equiv 2\mdot a-1\ \modu {\pi_i}$, $i=2,\dots,k$, and $b\equiv 0\ \modu 2$ due to the Chinese remainder theorem. We set $b_1=0$ and $b_i=2\mdot a_i\modo {\pi_i}$, $i=2,\dots,k$. Then\vspace{3.5mm}

\qquad $b+2\mdot x\equiv 0+2\mdot x\equiv 0\equiv b_1\ \modu 2$\qquad for $x=0,\dots,m$,\qquad and

\qquad $b+2\mdot x+1\equiv  2\mdot a-1+2\mdot x+1\equiv 2\mdot (a+x)\equiv 2\mdot a_i\equiv b_i\modu {\pi_i}$

\hspace{75mm} for $x=0,\dots,m-1$ and appropriate $i\ge2$.
\ \\\\
($\Leftarrow$)
Given $\forall\ y\in\{0,\dots,2\mdot m\}\ \exists\ i\in\{1,\dots,k\}:b_i\equiv b+y\modu {\pi_i}$.\\
Then $\forall\ x\in\{0,\dots,m-1\}\ \exists\ i\in\{1,\dots,k\}:b_i\equiv b+2\mdot x+1\modu {\pi_i}$.

There exist $a\in\mZ$ with $2\mdot a\equiv b+1$ and  $a_i\in\{0,\dots,\pi_i-1\}$ with $2\mdot a_i\equiv b_i\ \modu {\pi_i}$, $i=2,\dots,k$. For $x=0,\dots,m-1$, we get 
$$2\mdot(a+x)\equiv 2\mdot a+2\mdot x\equiv b+2\mdot x+1\equiv  b_i\equiv 2\mdot a_i\ \modu {\pi_i},$$
and thus $a+x\equiv a_i\ \modu {\pi_i}$ for appropriate $i\ge2$ because $2\ndiv\pi_i$.
\end{proof}
\

In continuation of  proposition\ref{coprime-covering} and corollary\ref{diff<-->rcov},
members of $D(k)$ can now be characterised by restricted coverings using only odd primes.

\begin{coro} \label{coprime-covering-2}\ \\
Let $m,k\in\mN$, $k\ge2$. Then\vspace{3.5mm}

\qquad $m\in D(k)\ \iff\ \exists\ a_i\in\{1,\dots,p_i-1\},\ i=2,\dots,k :$

\hspace{41mm}$\{a_i\modo p_i\}_{i=2}^k$ is a restricted covering of $ \ap 1 {m/2-1}$.
\end{coro}

\begin{proof}
Because of corollary\ref{diff<-->rcov}, $m\in D(k)$ if and only if there exists a restricted covering $\{b_i\modo p_i\}_{i=1}^k$ of $\ap 1 {m-1}$. This is equivalent to the existence of a restricted covering $\{a_i\modo p_i\}_{i=2}^k$ of $ \ap 1 {m/2-1}$ by  lemma\ref{two} and corollary \ref{non-zero}.
 \end{proof}
\

Consequently, we can reduce the computational effort by only considering\lb odd primes.

\subsection*{Algorithm}

The Greedy Permutation Algorithm (GPA) \cite{Ziller_Morack_2016} has proved to be efficient to compute Jacobsthal's function and related sequences. It was developed for the search of general coverings and considers only odd primes as well. 

 We extended and adapted this algorithm for the determination of restricted\lb coverings which represent the elements of the sets $D(k)$. The exclusion of one\lb additional residue class per prime was incorporated for this purpose.\\

The underlying idea of GPA is based on a specific order of choosing appropriate residue classes for each prime under consideration. GPA chooses that residue first which covers most of the free positions. It need not be reconsidered on the same recursion level again because all possible combinations are checked with its first use. GPA carries out a comprehensive search across all relevant residue classes. However, the greedy principle makes it possible to discard ineligible permutations of residue classes much earlier than by processing them in successive order.

GPA is a recursive algorithm. It starts with a given set of primes and an empty\lb array representing the sequence $\langle 1\rangle_m$ of a tentative length $m$. The algorithm tries\lb to find suitable residues $a_i\modo p_i$ such that as many positions of the sequence as\lb possible can be covered. In each step, one of the remaining $a_i\modo p_i$ and therefore also the prime $p_i$ itself are chosen, the corresponding array elements are assigned, and the number of still free positions is compared with the maximum number of positions which can be covered by the remaining primes. If there is no more chance to fill the sequence, then the current step will be skipped. A detailed description of the Greedy Permutation Algorithm can be found in \cite{Ziller_Morack_2016}.\\

Some simple generalisations of GPA make the efficient computation of restricted coverings possible. The tentative sequence length $m$ remains fixed. All residue classes $a_i\congr  m+1\modu p_i,\ i=2,\dots,k$ are excluded from processing.
The following\lb pseudocode \ref{AGPA} summarises the algorithm applied for our computations.

\pagebreak

\begin{algorithm}[!h]
\caption{\ Adapted Greedy Permutation Algorithm (AGPA).} \label{AGPA}
\begin{algorithmic} 
\Procedure{adapted\_greedy\_permutation}{arr,i,ftab}
   \If {restricted\_covering\_found}  break
   \EndIf
   \If {i<k-1}
      \If {i=1}
           \State fill\_frequency\_table\_of\_remainders(ftab)
           \State n\_empty=m	\Comment Starting number of free array positions
           \State n\_possible=count\_max\_possible\_covered\_positions(ftab)\\
			 				\Comment Starting number of maximum coverable positions
       \Else
         \State n\_empty=update\_free\_array\_positions(arr)
         \State n\_possible=update\_max\_possible\_covered\_positions(ftab)
      \EndIf

       \If {n\_possible$\ge$n\_empty}
          \State select\_appropriate\_$a_i$\_and\_$p_i$(ftab)
          \State arr1=arr; \ fill\_array(arr1,$a_i$,$p_i$)
          \State  ftab1=ftab; \ update\_frequency\_table\_of\_remainders(ftab1)
          \State adapted\_greedy\_permutation(arr1,i+1,ftab1)	\Comment Permutation level i+1
          \State delete\_frequency\_of\_$a_i$\_mod\_$p_i$(ftab)
          \State adapted\_greedy\_permutation(arr,i,ftab)	\Comment Permutation level i
      \EndIf
   \Else
      \State count\_array(arr)
      \If {sequence\_is\_covered}
          \State record\_covering
          \State restricted\_covering\_found=true
          \State break
      \EndIf
   \EndIf
\EndProcedure

\vspace{1.75mm} \hrule \vspace{1.75mm}

\State m=sequence\_length	\Comment Sequence length
\State arr=empty\_array		\Comment Sequence array
\State plist=[$p_2,\dots,p_k$]	\Comment Array of primes
\State i=1	\Comment Starting prime array index
\State ftab=empty\_table	\Comment Frequency table of remainders
\State restricted\_covering\_found=false	\Comment No covering found yet
\State adapted\_greedy\_permutation(arr,i,ftab)	\Comment Recursion
\State \Return restricted\_covering\_found	\Comment Result
\end{algorithmic}
\end{algorithm}

\clearpage


\begin{table}[!h]
  \centering
  \renewcommand{\arraystretch}{0.975}
  \setlength{\tabcolsep}{4mm}
\begin{tabular}{!\vl r|r|r|r|l|r !\vl}
  \noalign{\hrule height 2pt}
\rule{0pt}{14pt}$k$&$p_k$&$h(k-1)$&$N_{min}(k)$&\clg non-existent differences&$h(k)$\\[2pt]
  \noalign{\hrule height 2pt}
\rule{0pt}{14pt}1&2&-&2&  -&2\\
2&3&2&4&  -&4\\
3&5&4&6&  -&6\\
4&7&6&10&  -&10\\
5&11&10&14&  -&14\\
6&13&14&18&\clg 20&22\\
7&17&22&26&  -&26\\
8&19&26&30&\clg 32&34\\
9&23&34&40&  -&40\\
10&29&40&46&  -&46\\
11&31&46&58&  -&58\\
12&37&58&66&  -&66\\
13&41&66&74&  -&74\\
14&43&74&84&\clg 86,  88&90\\
15&47&90&96&\clg 98&100\\
16&53&100&106&  -&106\\
17&59&106&118&  -&118\\
18&61&118&132&  -&132\\
19&67&132&144&\clg146&152\\
20&71&152&164&\clg166, 170, 172&174\\
21&73&174&180&\clg182, 186, 188&190\\
22&79&190&192&\clg194&200\\
23&83&200&216&  -&216\\
24&89&216&228&\clg230&234\\
25&97&234&252&\clg254, 256&258\\
26&101&258&264&  -&264\\
27&103&264&276&\clg278&282\\
28&107&282&294&\clg296, 298&300\\
29&109&300&312&  -&312\\
30&113&312&326&\clg328&330\\
31&127&330&344&\clg346, 350, 352&354\\
32&131&354&366&\clg368, 370, 372, 376&378\\
33&137&378&384&\clg386&388\\
34&139&388&404&\clg406&414\\
35&149&414&422&\clg424&432\\
36&151&432&444&\clg446&450\\
37&157&450&466&\clg468, 472, 474&476\\
38&163&476&486&\clg488&492\\
39&167&492&510&  -&510\\
40&173&510&528&\clg530, 536&538\\
41&179&538&550&  -&550\\
42&181&550&570&\clg572&574\\
43&191&574&590&\clg592, 596, 598&600\\
44&193&600&616&  -&616\\[2pt]
  \noalign{\hrule height 2pt}
\end{tabular}
  \caption{Computational results.}
  \label{tab_res}
\end{table}

\clearpage

\section{Results and final remarks}

In an exhaustive exploration for primes $p_k$ up to $k=44$, we computed $N_{min}(k)$ as well as all non-existent differences $2\mdot n$ between consecutive coprimes to $p_k\#$ with\lb $N_{min}(k)<2\mdot n<N_{max}(k)$. The results are summarised in table\ref{tab_res} above. The data provided include

\qquad the index $k$ of the greatest prime considered,

\qquad the prime number $p_k$,

\qquad the primorial Jacobsthal function of $k-1$,

\qquad the minimum even number $N_{min}(k)$ according to definition\ref{N-min},

\qquad all even numbers below $N_{max}(k)$ which are not element of $D(k)$,

\qquad and the primorial Jacobsthal function of $k$, $h(k)=N_{max}(k)$.\\

The general question of lacking differences is open. For which $k\in\mN$ we\lb have $N_{min}(k)=N_{min}(k)=h(k)$, or otherwise $N_{min}(k)<N_{min}(k)$, i.e. there exists an even integer lower than $h(k)$ which cannot be represented as a difference between consecutive coprimes to $p_k\#$? For both variants, there might exist infinitely many $k$. The data presented do not make us prefer any of these options.\\

A simple lower bound on potential non-existent differences, however, was derived in corollary\ref{lower-bound}. De Polignac presumed $2\mdot p_{k-1}\le N_{min}(k)$\cite{de_Polignac_1849_II} for $k>1$. Instead, the results in table\ref{tab_res}  show that his assumption is wrong, at least for some small prime numbers. It is violated for k=6 and k=8. Perhaps these are the only exceptions because the quotient $N_{min}(k)/p_{k-1}$ appears to grow on average.

The data presented imply a different assumption for $k>1$: At least all even natural numbers through $h(k-1)$ occur as differences of coprimes to $p_k\#$. This is equivalent to the statement that any potential difference $2\mdot n<h(k-1)$ which does not occur as\lb differences of coprimes to $p_{k-1}\#$ occurs in case of coprimes to $p_k\#$. All computed data satisfy this assumption.

\begin{conj} \label{conj}\ \\
Let $k\in\mN>1$. Then
$$h(k-1)\le N_{min}(k).$$
\end{conj}

\begin{prop}\ \\
Let $k,n\in\mN$, $k>1$. Then\vspace{4mm}

\qquad $h(k-1)\le N_{min}(k)$

\qquad\qquad $\iff$

\qquad $(2\mdot n<h(k-1)\land 2\mdot n\not\in D(k-1))
\Longrightarrow 2\mdot n\in D(k).$
\end{prop}

\begin{proof}
The equivalence follows from proposition\ref{k-->k+1}.
\end{proof}

It can be presumed that Conjecture\ref{conj} is a stronger requirement than de Polignac's assumtion, except for small prime numbers. For all known values of Jacobsthal's\lb function \cite{Ziller_Morack_2016}, we have $2\mdot p_{k-1}<h(k-1)$ for $k>18$, \ i.e. $h(k)>2\mdot p_k$ for $k>17$.

\pagebreak

\subsubsection*{Contact}
marioziller@arcor.de
\ \\\\



\end{document}